\documentclass[12pt]{article}
\usepackage[utf8]{inputenc}

\usepackage{amsmath, wrapfig}
\usepackage{dsfont, a4wide, amsthm, amssymb, amsfonts, graphicx}
\usepackage{fancyhdr, xspace, psfrag, setspace, supertabular, color}
\usepackage[colorlinks]{hyperref}
\usepackage{bbm}
\usepackage{stmaryrd}
\usepackage{txfonts}
\usepackage{xcolor} 
\usepackage{enumerate}  

\usepackage{geometry}
\newtheorem{theorem}{Theorem}[section]
\newtheorem{proposition}[theorem]{Proposition}
\newtheorem{lemma}[theorem]{Lemma}
\newtheorem*{claim1*}{Claim 1}
\newtheorem*{claim2*}{Claim 2}
\newtheorem*{definition*}{Definition}
\newtheorem{corollary}[theorem]{Corollary}

\newtheorem{definition}[theorem]{Definition}

\newtheorem{question}[theorem]{Question}

\DeclareMathOperator{\sign}{sign}

\def\e{\epsilon}
\def\E{\mathop{\mathbb{E}}}

\def\R{\mathbb{R}}

\def\N{\mathbb{N}}
\def\P{\mathbb{P}}

\def\a{\alpha}
\def\be{\beta}

\def\g{\gamma}
\def\d{\delta}

\def\b1{\mathbbm{1}}

\newcommand{\appropto}{\mathrel{\vcenter{
  \offinterlineskip\halign{\hfil$##$\cr
    \propto\cr\noalign{\kern2pt}\sim\cr\noalign{\kern-2pt}}}}}

\def \sp#1{\langle #1\rangle}

\title{A counterexample to a strengthening of a question of Milman}
\author{W. T. Gowers and K. B. Wyczesany}
\date{ }

\begin{document}

\maketitle

\begin{abstract}
    Let $|\cdot|$ be the standard Euclidean norm on $\R^n$ and let $X=(\R^n,\|\cdot\|)$ be a normed space.  A subspace $Y\subset X$ is \emph{strongly $\a$-Euclidean} if there is a constant $t$ such that $t|y|\leq\|y\|\leq\a t|y|$ for every $y\in Y$,  and say that it is \emph{strongly $\a$-complemented} if $\|P_Y\|\leq\a$, where $P_Y$ is the orthogonal projection from $X$ to $Y$ and $\|P_Y\|$ denotes the operator norm of $P_Y$ with respect to the norm on $X$. We give an example of a normed space $X$ of arbitrarily high dimension that is strongly 2-Euclidean but contains no 2-dimensional subspace that is both strongly $(1+\e)$-Euclidean and strongly $(1+\e)$-complemented, where $\e>0$ is an absolute constant. This example is closely related to an old question of Vitali Milman.
\end{abstract}
    
\section{Introduction}

A famous theorem of Dvoretzky \cite{Dvoretzky1961} asserts that for every positive integer $k$ and every $\e>0$ there exists a positive integer $n$ such that every normed space of dimension at least $n$ has a subspace $Y$ of dimension $k$ such that $d(Y,\ell_2^k)\leq 1+\e$, where $d$ is the Banach-Mazur distance.  A highly influential second proof of Dvoretzky's theorem was given by Milman \cite{Milman1971},  which exploited measure concentration and led to many other arguments based on the same fundamental idea.

Let us call an $n$-dimensional normed space $X$ $C$-\emph{Euclidean} if $d(X,\ell_2^n)\leq C$. A fairly straightforward use of Milman's method yields the following statement.
\begin{theorem} \label{strongdvoretzky} For every $C>1$ and every $\e>0$ there exists $c>0$ such that for every $n\in\N$, every $n$-dimensional $C$-Euclidean normed space $X$ has a subspace $Y$ of dimension at least $cn$ that is $(1+\e)$-Euclidean. 
\end{theorem}
\noindent In other words, under the additional hypothesis that $X$ is $C$-Euclidean (and in fact under much weaker assumptions than this), one can obtain a linear dependence between the dimension of $Y$ and the dimension of $X$.

There is a large literature on finding `nice' subspaces of normed spaces under various conditions, but most of this literature pays little attention to how those subspaces sit in the main space. In particular, a desirable property for a subspace $Y\subset X$ is that it should be \emph{complemented}. In an infinite-dimensional context, one says that $Y$ is complemented if $Y=PX$ for a continuous projection $P$ on $X$. In a finite-dimensional context, we need a more quantitative definition: $Y$ is said to be $\a$-complemented if $Y=PX$ for a projection $P$ of operator norm at most $\a$.  

There are several open problems about the existence of  complemented subspaces. For example, it is not known whether there is a constant $C$ such that for every $k$ there exists $n_0$ such that for every $n\geq n_0$, every $n$-dimensional normed space has a $C$-complemented subspace of dimension at least $k$ and codimension at least $k$.  (For a partial result in this direction, see \cite{szarektomczak2009}.)

In this paper we consider the following question of Milman.
\begin{question} \label{milman}
Let $k\in\N$, let $C\in\R$, and let $\e>0$. Does there exist $n\in\N$ such that every $C$-Euclidean normed space $X$ of dimension at least $n$ has a $k$-dimensional subspace $Y$ that is $(1+\e)$-Euclidean and $(1+\e)$-complemented?
\end{question}
\noindent We do not answer the question, but we give a negative answer to a question that is sufficiently close to Milman's to suggest very strongly that Milman's question has a negative answer. 

To describe our result, we introduce two further definitions.  We shall write $|\cdot|$ for the standard Euclidean norm on $\R^n$.  Let us call a normed space $X=(\R^n,\|\cdot\|)$ \emph{strongly $\a$-Euclidean} if there is a constant $t$ such that $t|x|\leq\|x\|\leq\a t|x|$ for every $x\in X$, and let us call a subspace $Y$ of $X$ \emph{strongly $\a$-complemented} if the orthogonal projection $P_Y$ from $X$ to $Y$ has operator norm at most $\a$ (with respect to the norm $\|\cdot\|$ on $X$). The first definition is stronger than merely being $\a$-Euclidean, because instead of asking for any linear map $T$ such that $|Tx|\leq\|x\|\leq\a|Tx|$, we ask for $T$ to be a multiple of the identity, and the second is stronger than merely being $\a$-complemented because we ask for the projection to be orthogonal with respect to the standard inner product on $\R^n$.

These are natural strengthenings to consider, in the light of the fact that Milman's proof of Theorem \ref{strongdvoretzky} begins by observing that without loss of generality $X$ is strongly $C$-Euclidean and then proceeds to find a strongly $(1+\e)$-Euclidean subspace. Thus, one would expect Question \ref{milman} to have a positive answer if and only if the following question also has a positive answer.
\begin{question}\label{strongmilman}
Let $k\in\N$, let $C\in\R$, and let $\e>0$. Does there exist $n\in\N$ such that every strongly $C$-Euclidean normed space $X$ of dimension at least $n$ has a $k$-dimensional subspace $Y$ that is strongly $(1+\e)$-Euclidean and strongly $(1+\e)$-complemented?
\end{question}

Our main theorem is an example that shows that the answer to Question \ref{strongmilman} is negative.

\begin{theorem}\label{thm:counterexample}
There exist constants $\e>0$ and $C\in\R$ such that for all sufficiently large $n\in\N$ there is an $n$-dimensional normed space that is $C$-Euclidean but contains no 2-dimensional subspace that is both strongly $(1+\e)$-Euclidean and strongly $(1+\e)$-complemented.
\end{theorem}

\noindent The rest of the paper is devoted to the proof, apart from a few remarks in the concluding section.

\section{Good vectors}

We start with a definition that allows us to reformulate in a convenient way the condition that $Y$ is strongly $(1+\e)$-Euclidean and strongly $(1+\e)$-complemented.

We write $\sp{\cdot,\cdot}$ for the standard inner product.  

\begin{definition}\label{egooddef} Let $X=(\R^n,\|\cdot\|)$ be a normed space and let $x\in X$. We say that $x$ is $\e$-\emph{good} if 
\[\langle x,y\rangle\leq(1+\e)\frac{\|y\|}{\|x\|}|x|^2,\]
for every vector $y\in\R^n$.\end{definition}

To see what this means geometrically, consider the orthogonal projection $P_x$ on to the 1-dimensional subspace of $\R^n$ generated by $x$. Writing $x'$ for the normalized vector $x/|x|$, this has the formula
\[P_xy=\langle x',y\rangle x'.\]
Hence, the operator norm of $P_x$ (as a map from $X$ to $X$) is the maximum of the quantity
\[\frac{\langle x',y\rangle\|x'\|}{\|y\|}=\frac{\langle x,y\rangle\|x\|}{|x|^2\,\|y\|}\]
over all non-zero $y\in\R^n$. It follows that $x$ is $\e$-good if and only if $P_x$ has operator norm at most $1+\e$.  Since the definition of an $\e$-good point $x$ does not depend on the norm of $x$, it is enough to consider unit vectors.  Let us write $S^n=\{x\in \R^{n+1}: |x|=1\}.$

We now show that a subspace $Y$ of a space $X$ is strongly $(1+\e)$-Euclidean and strongly $(1+\e)$-complemented for some small $\e$ if and only if every $y\in Y$ is $\d$-good for some small $\d$. 

\begin{lemma}\label{lem:good=compl+eucl}
Let $X=(\R^n,\| \cdot \|)$ be a normed space and let $Y\subset X$ be a subspace. 
\begin{enumerate}
\item If $Y$ is strongly $(1+\e)$-complemented and strongly $(1+\e)$-Euclidean, then every $y\in Y$ is $(2\e+\e^2)$-good.
\item If $\e\leq 1/9\pi^2$ and every point in $Y$ is $\e$-good, then $Y$ is strongly $(1+\e)$-complemented and strongly $(1+3\pi\sqrt{\e})$-Euclidean. 
\end{enumerate}
\end{lemma}

Before we prove the statement, note that this characterization reduces Question \ref{strongmilman} to the following question.

\begin{question}\label{egoodquestion} Let $\e>0$, $C\geq 1$ and $k\in\N$. Does there exist $n$ such that if $\|\cdot\|$ is a norm on $\R^n$ such that $|x|\leq\|x\|\leq C|x|$ for every $x\in\R^n$, then the space $(\R^n,\|\cdot\|)$ has a subspace $Y$ of dimension $k$ such that every $y\in Y$ is $\e$-good?\end{question}

\begin{proof}[Proof of Lemma \ref{lem:good=compl+eucl}]
    Let $P_Y$ be the orthogonal projection onto $Y$. If $Y$ is strongly $(1+\e)$-Euclidean and strongly $(1+\e)$-complemented, then $\|P_Y x\| \le (1+\epsilon)\|x\|$ for every $x \in X$ and there exists $\lambda\in \mathbb{R}$ such that $ \lambda |y| \le \|y\| \le (1+\epsilon) \lambda |y|$ for every $y \in Y$. From this it follows that for every $y \in Y$ and every $x \in X$ we have 
    \[\left\langle y,x \right\rangle = \left\langle y,P_Y x \right\rangle  \le |y|\, |P_Y x| \le |y|\, \frac{1}{\lambda}  \| P_Y x\| \le (1+\epsilon) \lambda \frac{|y|^2}{\|y\|} \frac{1}{\lambda} (1+\epsilon) \|x\|  = (1+\epsilon)^2 \frac{\|x\|}{\|y\|}|y|^2, \]
    which implies that every point $y$ in $Y$ is $(2\epsilon + \epsilon ^2)$-good, as claimed.
    
    Conversely, assume that every point in $Y$ is $\epsilon$-good, so that for every $y \in Y$ and every $x \in X$ we have the inequality
    \[\langle y,x \rangle \le (1+\epsilon) \frac{\|x\|}{\| y\|} |y|^2.\]
    Choose $x\in X$. Then $P_Y x\in Y$, so
    \[|P_Yx|^2=\langle P_Yx,P_Yx\rangle=\langle P_Yx,x\rangle\leq (1+\e)\frac{\|x\|}{\|P_Yx\|} |P_Yx|^2,\]
and therefore $\|P_Yx\|\leq(1+\e)\|x\|$. It follows that $Y$ is strongly $(1+\e)$-complemented.   
    
     Now assume for a contradiction that the subspace $Y$ is not strongly $(1+a)$-Euclidean with $0<a$. In particular, this means that we can find two unit vectors $y,w \in Y$ such that $\|y\| = \|w\|(1 +a)$. Without loss of generality we may assume that $a\le 1/2$. 
     
    Let us consider a sequence of unit vectors $w=x_0, \,x_1 , \ldots , x_{m-1},\, x_m=y$ that are equally spaced along the shortest arc that joins $w$ to $y$ (which is unique, since $w$ cannot equal $-y$). By the pigeonhole principle there exists $i$ such that \[\|x_i\|(1+a)^{1/m} \le \|x_{i+1}\|.\]
         We shall choose $m$ to ensure that $x_i$ is a witness for $x_{i+1}$ not being $\e$-good. Indeed, if we assume that $m$ is at least $3\pi^2/a$ 
         then since the angle between $x_i$ and $x_{i+1}$ is at most $\pi/m$ we get that
\begin{align*}
    \langle x_{i+1},x_i\rangle\frac{\|x_{i+1}\|}{\|x_i\|\,|x_{i+1}|^2}&\geq\cos ( \angle x_i x_{i+1})(1+a)^{1/m}\geq\Bigl(1-\frac{\pi^2}{2m^2}\Bigr)\Bigl(1+\frac am - \frac{a^2}{2m}\Bigr)\\ &\geq1+\frac am-\frac{a^2m^2+\pi^2m+\pi^2a}{2m^3}\geq1+\frac a{2m}.
\end{align*}
    Here we used the fact that for $0<\gamma<1$ and $a>0$ we have that $(1+a)^\gamma \ge 1+\gamma a-\frac{\gamma(1-\gamma)}{2}a^2\ge  1+\gamma a-\frac{\gamma a^2}{2}$, and the assumptions that $0<a\le 1/2$ and $m\ge 3\pi^2/a$.
    
    It follows that the point $x_{i+1}$ is not $\frac a{2m}$-good. Therefore, if every point is $\e$-good, we must have that $\frac a{\lceil 6\pi^2/a\rceil}\leq\e$, which implies that $a\leq 3\pi\sqrt{\e}$. Thus, we find that $Y$ is strongly $(1+3\pi\sqrt{\e})$-Euclidean, which completes the proof. \end{proof}

Next, we give an equivalent condition for a point to be $\e$-good for some small $\e$.  Before we state the result, let us recall that a \emph{support functional} of a norm $\|\cdot\|$ at $x$ is any non-zero linear functional $f$ such that for every $y$ with $\|y\|\le \|x\|$ we have $f(y)\leq f(x)$.
Note that if the norm is differentiable, then writing $f(x)$ for $\|x\|$, we have that any multiple of $f'(x)$ is a support functional at $x$.

In the next proposition, we shall use the standard identification of $\R^n$ with its dual. That is, we identify a vector $z$ with the linear functional $y\mapsto\langle y,z\rangle$. 

\begin{proposition}\label{goodpoints}
Let $(X,\|\cdot\|)$ be a normed space and suppose that $|x|\leq\|x\|\leq C|x|$ for every $x\in X$. For every $\d>0$ there exists $\e>0$ such that if $x\in X$ is any  $\e$-good point, then there exist $y,z$ such that $|x|=|y|$, $|x-y|<\d|x|$, $z$ is a support functional for $y$, and $|y-z|<\d|x|$. Conversely, for every $\e>0$ there exists $\d>0$ such that $x$ is an $\e$-good point if there exist $y,z$ such that $|x-y|<\d|x|$, $z$ is a support functional for $y$, and $|y-z|<\d|x|$.
\end{proposition}

\begin{proof}
We shall do the second part first. Let $0<\e\leq 1$ and suppose that there exist $y,z$ such that $z$ is a support functional for $y$, and $|y-x|$ and $|z-y|$ are both at most $\d |x|$. 


Now let $w\in X$. Then
\[\langle w,x\rangle=\langle w,z\rangle+\langle w,y-z\rangle+\langle w,x-y\rangle\leq\langle w,z\rangle+2\d|w||x|.\]
But $z$ is a support functional for $y$, so 
\[\langle w,z\rangle\leq\|w\|\,\|z\|^*=\|w\|\frac{\langle y,z\rangle}{\|y\|}\]
We also have that
\[\|y\|\geq\|x\|-C|x-y|\geq\|x\|-C\d|x|\geq(1-C\d)\|x\|.\]
Finally, since $|x-z|\le 2\d |x|$ we have that $|z|\leq(1+2\d)|x|$, so
\[\langle y,z\rangle=\langle x,z\rangle+\langle y-x,z\rangle\leq |x||z|+\d|x||z|\leq(1+\d)(1+2\d)|x|^2.\]
Putting all this together, we find that
\[\langle w,x\rangle\leq\frac{(1+\d)(1+2\d)}{1-C\d}\frac{\|w\|}{\|x\|}|x|^2+2\d|w||x|\leq\Big(\frac{(1+\d)(1+2\d)}{1-C\d}+2C\d\Big)\frac{\|w\|}{\|x\|}|x|^2.\]
It can be checked that if we set $\d=\e/5C$, then the factor in brackets is at most $1+\e$.

For the other direction, assume that for all $y$ such that $|y|=|x|$ and $|x-y|<\delta |x|$ we have that $|y-z|>\delta|x|$, where $z$ is the support functional at $y$, chosen such that $|z|=|y|$.

We can assume that $|x|=1$ and that for every unit vector $y$ with $|y-x|<\d$, we have that $|y-z|\geq\d$. It follows that 
\[\langle y,z\rangle = 1-|y-z|^2/2\leq 1-\d^2/2\]
and therefore that
\[|z-\sp{y,z}y|^2=1-\sp{y,z}^2\geq \d^2-\d^4/4,\]
which implies that the component of $z$ orthogonal to $y$ has size at least $\sqrt 3\d/2\geq\d/2$. 

It follows that for any $\gamma<\delta$  we can find a path on the unit sphere that starts at $x$ and ends at a point at distance at least $\g$ from $x$ such that the norm $\|\cdot\|$ decreases at a rate of at least $\d/2$ along the path. This gives us a unit vector $\bar{y}$ such that $|\bar{y}-x|\leq\g$ and \[\|\bar{y}\|\leq\|x\|-\g\d/2\leq\|x\|(1-\g\d/2C).\]

It follows that $\langle x,\bar{y}\rangle> 1-\gamma^2/2$, so
\[\langle x,\bar{y}\rangle>\frac{(1-\gamma^2/2)}{(1-\g\d/2C)}\frac{\|\bar{y}\|}{\|x\|}|x|^2.\]
Setting $\g=\d/2C$, we deduce that $x$ is not $\d^2/8C^2$-good.
\end{proof}

\section{Definition of the norm and an important observation}


The norm has a fairly simple definition. Let $P$ be a random orthogonal projection of rank $n/2$ and let $A=I+P$. Then we define 
\begin{align}\label{definition of the norm}
    \|x\|=\langle x,Ax\rangle^{1/2}+\eta n^{-1/2}\|x\|_1,
\end{align}
where $\eta>0$ is an absolute constant to be chosen later. (Note that $|x|\leq\|x\|\leq(\sqrt 2+\eta)|x|$, so as long as $\eta\leq 2-\sqrt 2$, this norm is strongly 2-Euclidean.) The first part of this norm is a weighted $\ell_2$ norm with respect to a random orthonormal basis, where half the weights are 2 and half are 1, and the second is a multiple of the standard $\ell_1$ norm.  Our  aim now is to prove that with probability greater than zero (and in fact close to 1) there is no 2D subspace that consists entirely of $\e$-good points, for some absolute constant $\e>0$.  That is, we will prove Theorem \ref{thm:counterexample} and therefore give a negative answer to Question \ref{strongmilman}.

The next lemma tells us what the support functionals are at a vector $x$. Let us use the notation $\sign(t)$ for the multivalued function from $\R$ to $\R$ that takes $t$ to 1 if $t>0$, to -1 if $t<0$, and to any element of $[-1,1]$ if $t=0$. Then if $x\in\R^n$ we write $\sign(x)$ for the result of applying the multivalued function $\sign$ pointwise. Let us also write $\|x\|_A$ for $\langle x,Ax\rangle^{1/2}$.

\begin{lemma}\label{supportfnl}
The support functionals at $x$ are multiples of $\frac{Ax}{\|x\|_A}+\eta n^{-1/2}\sign(x)$.
\end{lemma}

\begin{proof}
Essentially this is just a question of calculating the derivative of the norm, except that where the derivative is not defined we may have to give it several values (just as one might say that the derivative of $|x|$ at zero is any element of $[-1,1]$). 

Let $y$ be a sufficiently small vector. Then for any possible choice of $\sign(x)$, we have that
\[\|x+y\|\geq\|x\|+\frac{\langle Ax,y\rangle}{\|x\|_A}+\eta n^{-1/2}\langle\sign(x),y\rangle+o(y).\]
Therefore, if $y$ is orthogonal to some value of $\frac{Ax}{\|x\|_A}+\eta n^{-1/2}\sign(x)$, we have that $\|x+y\|\geq\|x\|+o(y)$, from which it follows easily that $\frac{Ax}{\|x\|_A}+\eta n^{-1/2}\sign(x)$ is a support functional at $x$.
\end{proof}


\begin{corollary}\label{near4d}
Let $Q$ be such that $P+Q=I$,  let $\e$ be sufficiently small and let $Y$ be a 2-dimensional subspace that consists entirely of $\e$-good points. Then for every unit vector $x\in Y$ there exists a unit vector $y$ with $|x-y|\le \delta$ and a value of $\sign(y)$ such that 
\[d\big(\eta n^{-1/2}\sign(y),PY+QY\big)\leq\d,\] 
where $\d$ tends to zero with $\e$. 
\end{corollary}

\begin{proof}
By Lemma \ref{supportfnl} and Proposition \ref{goodpoints} we find that if a unit vector $x$ is $\e$-good, then there exists a unit vector $y$, a value of $\sign(y)$, and a scalar $\mu$ such that $|y-x|<\d$ and 
\[\Big| \,\mu \Big(\frac{Ay}{\|y\|_A}+\eta n^{-1/2}\sign(y)\Big)-y\Big|<\delta/3,\]
where $\d>0$ tends to zero with $\e$. Since $|Ay|\leq\sqrt 2\|y\|_A$ and $n^{-1/2}|\sign(y)|\leq 1$, our assumption that $\sqrt 2+\eta\leq 2$ implies that $\mu\geq (1-\d/3)/2$.  For sufficiently small $\e$, we therefore have that $\mu\geq 1/3$.  

Recalling that $A=I+P$ and rearranging, we obtain the inequality
\[\Big(\frac 1\mu-\frac 2{\|y\|_A}\Big)Py + \Big(\frac 1\mu-\frac 1{\|y\|_A}\Big)Qy \, \approx_\delta \, \eta n^{-1/2} \sign(y),
\]
where we write $u\approx_\delta v$ to mean that $|u-v|\leq\delta$.  (We shall use this convenient notation throughout the rest of this paper.) \end{proof}


Corollary \ref{near4d} tells us that if we have a 2-dimensional subspace $Y$ that consists entirely of $\e$-good points, then every unit vector $x\in Y$ is close to a vector $y$ such that $\sign(y)$ is close to the subspace $PY+QY$, which has dimension at most 4. This is the main observation we shall use to obtain a contradiction.

\section{Outline of the proof and some technical lemmas}\label{outline}

Let us call a non-zero vector a \emph{sign vector} if all its coordinates have the same absolute value. As before, we write $X$ for $(\R^n,\|\cdot\|)$, though sometimes we abuse notation and use $X$ to refer simply to the vector space $\R^n$.

In order to show that the norm defined in \eqref{definition of the norm}  indeed constitutes a counterexample to Question \ref{strongmilman}, that is, that there is no two-dimensional subspace of $(\R^n, \| \cdot \|)$ that consists entirely of $\e$-good points, we shall obtain a contradiction using more precise versions of the following statements.

\begin{enumerate}
    \item Every $\e$-good point is close to $PX$ or $QX$.
    \item With high probability, no point that is close to $PX$ or $QX$ can be approximated by a point with only a few distinct coordinates.
    \item If $Y$ is a 2-dimensional subspace that consists entirely of $\e$-good points, then for every $x\in Y$ there exists $x'$ close to $x$ such that $\sign(x')$ is close to the subspace $PY+QY$.
    \item Using the first two statements, we deduce that the vectors $\sign(x')$ are not approximately contained in a 4-dimensional subspace. 
\end{enumerate}

Corollary \ref{near4d} is our precise version of Statement 3.  Let us now prove Statement 1, which is also fairly simple.

\begin{lemma}\label{closetoespace}  Let $x$ be an $\e$-good vector in $(X,\|\cdot\|)$, such that $|x|=1$. Then either $d(x,PX)\leq 3\d+2\eta$ or $d(x,QX)\leq 3\d+2\eta$, where $\d$ is given by Proposition \ref{goodpoints}.
\end{lemma} 

\begin{proof}
From the beginning of the proof of Corollary \ref{near4d} we obtain a unit vector $y$ such that $|x-y|\leq\d$ and such that 
\[\frac{Ay}{\|y\|_A}+\eta n^{-1/2}\sign(y)\approx_\delta \lambda y,\]
where we have written $\lambda$ for $1/\mu$ and used the fact that $\mu\geq 1/3$.

We have that 
\[d(x,PX)\le d(x,y)+d(y,PX)\le \d + |y-Py|=\d +|Qy|\] and similarly for $d(x,QX)$. Hence, our goal is to bound $\min\{ |Py|,|Qy| \}$. 

But $|Ay|\geq\|y\|_A$, and $\eta n^{-1/2}|\sign(y)|\leq\eta$, so $\lambda\geq 1-\eta-\d$ and
\[\left|Ay-\lambda\|y\|_A y\right|<(\d+\eta)\|y\|_A \le \sqrt{2}(\d+\eta).\]
Thus, $y$ is an approximate eigenvector of $A$ and it remains to prove that an approximate eigenvector of $A$ must be close to an eigenvector. (This is of course false for general linear maps.)

Since $P+Q=I$, we have $y=Py+Qy$  and $Ay=2Py+Qy$, so if $\nu$ is any scalar, then 
\[|Ay-\nu y|^2=(2-\nu)^2|Py|^2+(1-\nu)^2|Qy|^2.\]
Writing $2-\nu=a+1/2$ and $1-\nu=a-1/2$, one can rewrite the right-hand side as
\[\Big(a+\frac{|Py|^2-|Qy|^2}2\Big)^2+\frac 14\Big(1-(|Py|^2-|Qy|^2)^2\Big),\] 
from which we see that if $|Ay-\nu y|^2\leq\tau$ 
then 
\[(|Py|^2-|Qy|^2)^2\geq 1-4\tau,\]
which implies that either $|Py|^2\leq 2\tau$ or $|Qy|^2\leq 2\tau$. 
In our case, we may set $\tau= 2(\d+\eta)^2$, so $\min\{|Py|,|Qy|\}\leq 2(\d+\eta)$, which gives us the bound stated.
\end{proof}


Next, we formulate and prove a suitable version of Statement 2. We begin with a crude upper bound for the volume of the $\g$-expansion of the unit sphere of a subspace of dimension $cn$. (Much more accurate estimates exist, but for us a simple argument suffices.)

\begin{lemma}\label{subspacevol}
Let $m=cn$ and let $\g>0$. Assume that $2^{n+1}\g\geq 1$. Then the probability that a random unit vector has distance at most $\g$ from a given subspace $Y\subset X$ of dimension $m$ is at most $24^n\g^{n-m}$.
\end{lemma}

\begin{proof}
A spherical cap in $S^{n-1}$ of Euclidean radius $2\g$ has volume at most $(4\g)^{n-1}\leq 8^n\g^n=(8\g)^n$. By standard estimates, we can also find a $\g$-net of $Y$ of cardinality at most $(3/\g)^m$. But every point of $S^{n-1}$ that is within $\g$ of $Y$ is within $2\g$ of a point in the $\g$-net, and from this the result follows.
\end{proof}


\begin{lemma}\label{nobadpoints}
Let $k$ be a positive integer, let $c,\g>0$, and let $Y$ be a random subspace of $\ell_2^n$ of dimension $m$. Then the probability that $Y_\gamma$ contains a unit vector $x$ with at most $k$ distinct coordinates is at most $(3/\g)^k(48k)^n\g^{n-m}$.
\end{lemma} 

\begin{proof}
The number of partitions of $\{1,2,\dots,n\}$ into $k$ sets is at most $k^n$, and for each partition $E_1,\dots,E_k$ the set of vectors that are constant on each $E_i$ is a $k$-dimensional subspace, so there is a $\g$-net of the unit sphere of this subspace of size at most $(3/\g)^k$.

The probability that $Y_\g$ contains a vector with at most $k$ distinct coordinates is at most the probability that $Y_{2\g}$ contains a point in one of these $\g$-nets, which is at most 
\[(3/\gamma)^k (24k)^n (2\gamma)^{n-m}\le (3/\gamma)^k (48k)^n \gamma^{n-m}\]
by Lemma \ref{subspacevol} and a union bound. 
\end{proof}

We present one more technical lemma that is similar to Lemma \ref{nobadpoints}, and which will be an important part of the argument. Again we make no attempt to optimize bounds. 

\begin{lemma}\label{largesupport}
Let $Z$ be a random subspace of $\ell_2^n$ of dimension $m$. Then the probability that $Z_\g$ contains a point with support size at most $r$ is at most $288^n\g^{n-m-r}$
\end{lemma}

\begin{proof}
The number of sets of size at most $r$ is $\binom nr\leq 2^n$. 
For each such set $E$ the size of a $\g$-net of the unit sphere of the space of vectors supported on $E$ is at most $(3/\g)^{r}$, and for each point in such a net the probability that it is in $Z_{2\g}$ is at most $24^n(2\g)^{n-m}$. Therefore, the probability we wish to bound is at most
\[2^n(3/\g)^{r}24^n(2\g)^{n-m}\leq 288^n\g^{n-m-r},\]
which proves the lemma. 
\end{proof}

The key point we shall need from the above lemma is that for any $c>0$ there exists $\g>0$ such that if $n-m-r\geq cn$, then the probability that $Z_\g$ contains a point with small support is small. In particular, we have the following statement.

\begin{corollary}\label{pickgamma}
Let $\g=2^{-37}$. Let $X=\ell_2^n$ with $n\ge 2$, and let $P:X\to X$ be a random orthogonal projection of rank $n/2$ and let $Q=I-P$. Then the probability that either $PX_{2\g}$ or $QX_{2\g}$ contains a vector of support size at most $n/4$ is at most $\left( \frac{2}{3}\right)^{n}$. 
\end{corollary}

\begin{proof}
Applying Lemma \ref{largesupport}, we find that the probability that $PX_{2\g}$ contains a vector of support size at most $n/4$ is at most $288^n(2\g)^{n/4}=(288/512)^n$. The same is true of $QX_{2\g}$ and the result follows with room to spare.
\end{proof}

For the remainder of the paper, we shall assume that $P$ has been chosen in such a way that neither $PX_{2\g}$ nor $QX_{2\g}$ contains a vector of support size at most $n/4$, and neither $PX_\g$ nor $QX_\g$ contains a vector with at most five distinct coordinates. By Lemma \ref{nobadpoints} and Corollary \ref{pickgamma} such a $P$ exists.

\section{The set of signs cannot be squeezed into a 4-dimensional subspace}

Before we move to the heart of the argument, which will be a precise version of Statement 4, let us remark that as we move forward we shall be dealing with many parameters. Since we do not wish to choose them straight away we make sure that it is easy to keep track of all the dependencies by stating them clearly and giving each one a label.

If $Y$ is a 2-dimensional subspace that consists entirely of $\e$-good points, then the last formula in the proof of Corollary \ref{near4d} gives us for each unit vector $x\in Y$ a unit vector $y$ with $|x-y|\leq\d$ and coefficients $\a_y$ and $\be_y$ such that
\[n^{-1/2}\sign(y)\approx_{\d/\eta}\a_yPy+\be_yQy,\]
where $\a_y=\eta^{-1}(\lambda-2/\|y\|_A)$ and $\be_y=\eta^{-1}(\lambda-1/\|y\|_A)$. 

Recall also from the beginning of the proof of Lemma \ref{closetoespace} that we also have the equivalent formula
\[\eta n^{-1/2}\sign(y)\approx_\d\lambda y-\frac{Ay}{\|y\|_A},\]
from which it follows that 
$|\lambda|\in[1-\d-\eta,\sqrt 2+\d+\eta]$,  since $\|y\|_A\leq |Ay|\leq\sqrt 2\|y\|_A$. Therefore, provided that 
\begin{align}\label{cond2}
    \d+\eta\leq 2-\sqrt 2
\end{align}
it follows that $\lambda\in[\sqrt 2-1,2]$. In particular, it follows that 
$|\a_y|$ and $|\be_y|$ are at most $2\eta^{-1}$. This bound will be important later.  
Suppose now that 
\begin{align}\label{cond3}
    3\delta+2\eta\le\g.
\end{align}
If $x\in Y$, then by Lemma \ref{closetoespace}, either $d(x,PX)$ or $d(x,QX)$ is at most $3\d+2\eta$. Therefore, for every $y\in Y_\g$, either $d(y,PX)$ or $d(y,QX)$ is at most $3\d+2\eta+\g\le 2\g$, so by the assumption made at the end of the previous section, $y$ has support size at least $n/4$. That is, every vector in $Y_\g$ has support size at least $n/4$. We shall use this property frequently in the rest of the section.

Now let us choose non-negative real numbers $r_1,\dots,r_n$ and phases $\phi_1,\dots,\phi_n\in[0,2\pi)$ such that each unit vector in $Y$ is equal to 
\[x(\theta)=\Big(r_1\sin(\theta+\phi_1),\ldots ,r_n\sin(\theta+\phi_n)\Big)\]
for some $\theta\in[0,2\pi)$. Note that by looking at $\E_\theta|x(\theta)|^2$ we find that $\sum_ir_i^2=2$. 

\begin{lemma}\label{signcontinuity}
Let $\d,\xi>0$ $x$ and $y$ be two vectors in $\R^n$ such that $|x-y|\leq\d$. Then the number of $i$ such that $|x_i|\geq\xi n^{-1/2}$ and $\sign(x_i)\ne\sign(y_i)$ is at most $\xi^{-2}\d^2n$.
\end{lemma}

\begin{proof}
For each such $i$ we have that $|x_i-y_i|^2\geq\xi^2n^{-1}$, and our hypothesis is that $\sum_i|x_i-y_i|^2\leq\d^2$.
\end{proof}

Let $\a>0$ be a constant to be chosen later, let $E=\{i: \, r_i \ge \alpha n^{-1/2}\}$, and write $P_E$ for the coordinate projection on to $E$. It will also be convenient to write $E_0$ for $\{1,2,\ldots, n\}\setminus E$, which we think of as the set of coordinates where $Y$ almost vanishes.

\begin{lemma}\label{typical}
Let $0<c,\xi<1$ and let $\theta$ be chosen uniformly at random from $[0,2\pi)$. Then with probability at least $1-\xi/\a c$, the number of $i\in E$ such that $|x(\theta)_i|<\xi n^{-1/2}$ is less than $c|E|$.
\end{lemma}

\begin{proof}
Since $r_i\geq\a n^{-1/2}$ for every $i\in E$, and $x(\theta)_i=r_i\sin(\theta+\phi_i)$, we have that for each $i\in E$,
\[\P\Big[|x(\theta)_i|<\xi n^{-1/2}\Big]\le \P \Big[ |\sin(\theta+\phi_i)|< \xi/\alpha \Big]< \xi/\a.\]
To see the last inequality, note that $|\sin(\theta+\phi)|$ on the interval $[0,2\pi)$ has the same distribution as $\sin\theta$ on the interval $[0,\pi/2)$. But if $\theta$ belongs to that interval, then $\sin\theta\geq2\theta/\pi$ with equality only at 0 and $\pi/2$, so $\sin\theta<\xi/\a$ only if $\theta<\pi\xi/2\a$, which is true with probability less than $\xi/\a$. 
Therefore, the expected number of $i\in E$ such that $|x(\theta)_i|<\xi n^{-1/2}$ is less than $\xi|E|/\a$.

The result now follows from Markov's inequality.
\end{proof}

\iftrue
\else
\begin{lemma}\label{typical}
Let $c,\xi>0$ and let $\theta$ be chosen uniformly at random from $[0,2\pi)$. Then with probability at least $1-\pi\xi/2\a c$,
\[\Big|P_E\Big(\sign(x(\theta))-\sign(y)\Big)\Big|\leq 2(c+\xi^{-2}\d^2)^{1/2}n^{1/2}\]
for every $y$ such that $|x(\theta)-y|\leq\d$. 
\end{lemma}

\begin{proof}
Let $\theta$ be a random element of $[0,2\pi)$ and let $y$ be such that $|x(\theta)-y|\leq\d$. Let $A$ be the set of coordinates in $E$ on which $x(\theta)$ and $y$ have different sign. Then 
\[|x(\theta)-y|^2\geq\sum_{i\in A}r_i^2x(\theta)_i^2.\]
Also, since $r_i\geq\a n^{-1/2}$ for every $i\in E$, and $x(\theta)_i=r_i\sin(\theta+\phi_i)$, we have that for each $i\in E$,
\[\P\Big[|x(\theta)_i|\leq\xi n^{-1/2}\Big]\leq\pi\xi/2\a,\] 
so by Markov's inequality the probability that there are at least $c|E|$ such coordinates is at most $\pi\xi/2\a c$.

The number of coordinates such that $|x(\theta)_i|\geq\xi n^{-1/2}$ and $\sign(x(\theta))_i\ne\sign(y)_i$ is at most $\d^2n/\xi^2$, since $|x-y|\leq\d$. It follows that with probability at least $1-\pi\xi/2\a c$, 
\[|A|\leq(c+\xi^{-2}\d^2)n.\]
But 
\[\Big|P_E\Big(\sign(x(\theta))-\sign(y)\Big)\Big|=2|A|^{1/2},\]
so the result follows.
\end{proof}
\fi

For choices of $\xi$ and $c$ that we shall make later, let us call $\theta$ and $x(\theta)$ \emph{typical} if 
\begin{enumerate}[(i)]
    \item the number of $i\in E$ such that $|x(\theta)_i|<\xi n^{-1/2}$ is less than $c|E|$ (that is, the conclusion of Lemma \ref{typical} holds), and 
    \item there is no $i\in E$ with $\theta+\phi_i\in\{0,\pi\}$. 
\end{enumerate}
The second condition, which is there for convenience, holds with probability 1 and ensures that $\sign(x(\theta))_i\in\{-1,1\}$ for every $i\in E$. Later we shall want to be sure that typical vectors exist, for which Lemma \ref{typical} tells us that a sufficient condition is the inequality
\begin{align}\label{cond4}
    \xi<\alpha c.
\end{align}

Let us now define $\Sigma$ to be the set of all vectors of the form $n^{-1/2}P_E\sign(x(\theta))$ such that $x(\theta)$ is a typical element of $Y$, let $\beta>0$ be another parameter to be chosen later, and let $2k$ be the size of a maximal centrally symmetric $\be$-separated subset of $\Sigma$ (so $V$ consists of $k$ antipodal pairs). Note that $V$ is a $\be$-net of $\Sigma$.

Since $Y\subset PX_\gamma$ or $Y\subset QX_{\gamma}$  and we chose $P$  in such a way that both $PX_{2\gamma}$  and $QX_{2\gamma}$ does not contain a vector of support size at most $n/4$, we have that every vector in $Y$, and even in $Y_\gamma$,  has support size at least $n/4$.  Assuming that
\begin{align}\label{cond3a}
\a\leq\g,
\end{align}
it follows that for every $y\in Y$ its set of ``large'' coordinates $|E|$ has cardinality at least $n/4$, since $|y-P_Ey|<\a$ for every $y\in Y$ and,  clearly,  $P_Ey$ has support size $|E|$. 
Thus, $\Sigma$ consists of vectors in a sphere of radius $n^{-1/2}|E|^{1/2}\geq 1/2$, so as long as 
\begin{align}\label{cond5} \be\leq 1\end{align}
we can choose any typical vector $x(\theta)$, let $v=n^{-1/2}P_E\sign(x(\theta))$, and thereby obtain a $\be$-separated subset $\{v,-v\}$ of $\Sigma$. This proves that $k\geq 1$.

We now consider three cases depending on the size of $V$.

\subsection{Case 1: $k=1$.}

Let $\zeta=\xi/\a c$ and let $V=\{v,-v\}$. Since $\theta$ is typical with probability at least $1-\zeta$, every closed interval of length greater than $2\pi\zeta$ contains a typical $\theta$. 
If $n^{-1/2}P_E\sign(x(\theta))\approx _\be v$, then $n^{-1/2}P_E\sign(-x(\theta)) \approx _\beta -v$, so there is at least one $\theta$ such that $n^{-1/2}P_E\sign(x(\theta)) \approx _\be v$ and one such that $n^{-1/2}P_E\sign(x(\theta)) \approx _\be -v$. 

Let $A$ be the set of typical $x(\theta)$ such that $n^{-1/2}P_E\sign(x(\theta))\approx _\be v$ and note that $-A$ is the set of typical $x(\theta)$ such that $n^{-1/2}P_E\sign(x(\theta))\approx _\be -v$. Note also that every typical $x(\theta)$ belongs to $A$ or $-A$, and that therefore neither $A$ nor $-A$ is empty. Writing $B$ for the set of points at distance at most $\pi\zeta$ from $A$, we also have that $B$ and $-B$ are closed and that $B\cup -B$ is the entire unit circle of $Y$. To see the last assertion, let $x(\theta_0)$ be a point in the unit circle of $Y$. Then the closed interval of length $2\pi\zeta$ centred at $\theta_0$ contains a typical point $\theta$, so $x(\theta)$ belongs to $A\cup(-A)$ and therefore $x(\theta_0)\in B\cup -B$, as the distance from $x(\theta_0)$ to $x(\theta)$ is at most $\pi\zeta$. 
Since the unit circle is connected, $B\cap -B$ is non-empty, from which it follows that there exist typical unit vectors $x,y\in Y$ such that $|x-y|\leq 2\pi\zeta$ and such that $n^{-1/2}P_E\sign(x)\approx_\be v$, and $n^{-1/2}P_E\sign(y)\approx_\be-v$. 

It follows that $n^{-1/2}P_E\sign(x)$ differs from $v$ in at most $\be^2n/4$ coordinates, and $n^{-1/2}P_E\sign(y)$ differs from $-v$ in at most $\be^2n/4$ coordinates. Therefore, $P_E\sign(x)$ and $P_E \sign(y)$ are equal in at most $\be^2n/2$ coordinates.  
Moreover, by Lemma \ref{signcontinuity}, the number of $i$ for which $\sign(x_i)\ne\sign(y_i)$ and $|x_i|\geq \rho n^{-1/2}$ 
is at most $\rho^{-2}(2\pi\zeta)^2n=\left(\frac{2\pi\xi}{\alpha c\rho}\right)^2n$.

From these two facts it follows that the number of coordinates $i\in E$ for which $|x_i|\geq\rho n^{-1/2}$ is at most $\big(\left(\frac{2\pi\xi}{\alpha c\rho}\right)^2+\frac{\be^2}2\big)n$. 
Therefore, we find that $x$ has distance at most $\rho$ from a vector of support size at most $\big(\left(\frac{2\pi\xi}{\alpha c\rho}\right)^2+\frac{\be^2}2\big)n$.

If we choose parameters in such a way that 
\begin{align}\label{cond6}
\rho \le \g
\end{align}

and 
\begin{align}\label{cond7}
    \left(\frac{2\pi\xi}{\alpha c\rho}\right)^2+\frac{\be^2}2 <1/4
\end{align}
we obtain a contradiction with the fact that $Y_\g$ does not contain a vector of support size less than $n/4$.

\subsection{Case 2: $2\leq k\leq 4$}

Let $V=\{\pm v_1,\dots,\pm v_k\}$. Since each $v_i$ is of the form $n^{-1/2}P_E\sign(x(\theta))$ for some typical vector $x(\theta)\in Y$, it takes values $\pm n^{-1/2}$ in $E$. 

We now show that either this case can be reduced to the case $k=1$ with $\be$ replaced by $48\be$ or there is a subset $V'$ of $V$ consisting of at least two antipodal pairs such that $V'$ is a $3\kappa$-separated $\kappa$-net of $\Sigma$ and $\be\le \kappa \le 16\be$.

Since $V$ is a $\be$-net of $\Sigma$, then if it is $3\be$ separated then we are done. If not, we can find $i\ne j$ such that $|v_i-v_j|\leq 3\be$. 
Then we can remove $\pm v_j$ from $V$ and we will still have a $4\be$-net. Similarly, if $V'=V\setminus \{ \pm v_j\} $ is $12\be$-separated we are done, but if it contains two distinct elements $v_i,v_j$ such that $|v_i-v_j|\leq 12\be$, then again we can remove $\pm v_j$ and we will still have a $16\be$-net. Finally, if there are two distinct elements $v_i,v_j$ with $|v_i-v_j|\leq 48\be$, 
then we may remove $\pm v_j$ and end up with $V'$ of the form $\{v,-v\}$ and we are back in case $k=1$. However, now $\be$  is replaced by $48\be$, which we must allow for when choosing our parameters, so we need to strengthen condition (\ref{cond5}) to the condition
\begin{align}\label{cond8}
    \beta \le 1/48.
\end{align}

If the process stops before we reach $k=1$, then we have a $3\kappa$-separated $\kappa$-net $V'=\{\pm v_1,\dots,\pm v_m\}$ of $\Sigma$ such that $2\leq m\leq 4$  and $\kappa\leq 16\be$ as claimed. 

Recall that every interval of length greater than $2\pi\zeta$ contains a typical $\theta$, and hence by a connectedness argument similar to the one used for the case $k=1$ there must exist $\theta,\phi$, 
and $v_i\ne \pm v_j$ such that 
\[|\theta- \phi|\leq 2\pi\zeta,\] 
\[|n^{-1/2}P_E\sign(x(\theta))-v_i|\leq\kappa,\] 
and 
\[|n^{-1/2}P_E\sign(x(\phi))-v_j|\leq\kappa.\] 
Since $|v_i\pm v_j|\geq 3\kappa$, it follows that $n^{-1/2}|P_E\sign(x(\theta))-P_E\sign(x(\phi))|$ and $n^{-1/2}|P_E\sign(x(\theta))+P_E\sign(x(\phi))|$ are both at least~$\kappa$.
Now recall that for every $x\in Y$ we can find $y$ with $|x-y|\le \delta$ such that 
\[n^{-1/2}\sign(y)\approx_{\d/\eta}\a_yPy+\be_yQy\]
for some constants $\a_y$ and $\be_y$. Let us choose $y(\theta)$ and $y(\phi)$ that have this relationship with $x(\theta)$ and $x(\phi)$, respectively. 

By Lemma \ref{signcontinuity} and the assumption that $x$ is typical, the number of coordinates in $E$ such that $\sign(x(\theta))$ and $\sign(y(\theta))$ differ is at most $c|E|+\xi^{-2}\d^2n\leq(c+\xi^{-2}\d^2)n$, so 
\[n^{-1/2}|P_E\sign(x(\theta))-P_E\sign(y(\theta))|\leq 2(c+\xi^{-2}\d^2)^{1/2},\]
and similarly for $\phi$. If we choose parameters in such a way that
\begin{align}\label{cond9}
    c+\xi^{-2}\d^2\leq\be^4/256,
\end{align}
it follows that these distances are both at most $\be^2/8$.

\iftrue
\else
, and therefore that 
\[n^{-1/2}|\sign(y(\theta))-\sign(y(\phi))|\geq\be/3\]
and
\[n^{-1/2}|\sign(y(\theta))+\sign(y(\phi))|\geq\be/3.\]
Speaking more loosely, this shows that $\sign(y(\theta))$ and $\sign(y(\phi))$ are not roughly proportional to each other.
\fi

Let us write $\alpha(\theta)$ instead of $\alpha_{y(\theta)}$, and similarly for $\phi$.  Then
\[n^{-1/2}\sign(y(\theta))\approx_{\d/\eta}\alpha(\theta) Py(\theta)+\beta(\theta) Qy(\theta),\]
and
\[n^{-1/2}\sign(y(\phi))\approx_{\d/\eta}\alpha(\phi) Py(\phi)+\beta(\phi) Qy(\phi).\]

Also, since $|\theta-\phi|\leq 2\pi\zeta$, we have 
$|x(\theta)-x(\phi)|\leq 2\pi\zeta$ (because $|x(\theta)-x(\phi)|<|\theta-\phi|$), 
and therefore $|y(\theta)-y(\phi)|\leq 2\pi\zeta +2\d$. It follows that
\[|\alpha(\phi) Py(\phi)+\beta(\phi) Qy(\phi)-(\alpha(\phi) Py(\theta)+\beta(\phi) Qy(\theta))|\leq 2(\pi\zeta+\d)(|\alpha(\phi)|+|\beta(\phi)|).\] 
Now recall that $|\alpha(\phi)|$ and $|\beta(\phi)|$ are both at most $2\eta^{-1}$. We therefore obtain the approximation 
\[n^{-1/2}\sign(y(\phi))\approx_{\sigma}\alpha(\phi) Py(\theta)+\beta(\phi) Qy(\theta),\]
where
    $\sigma=\d/\eta+8(\pi\zeta+\d)/\eta=(8\pi\zeta+9\d)/\eta.$

It is convenient to encapsulate our knowledge so far as an approximate matrix equation
\[\begin{pmatrix}\alpha(\theta)&\beta(\theta)\\ \alpha(\phi)&\beta(\phi)\\ \end{pmatrix}\begin{pmatrix}Py(\theta)\\ Qy(\theta)\\ \end{pmatrix}\approx_\sigma n^{-1/2}\begin{pmatrix}\sign(y(\theta))\\ \sign(y(\phi))\end{pmatrix}\]
where by $\approx_\sigma$ in this context we mean that the approximation holds coordinatewise. 

The rough idea of what we shall now do is as follows. Because $\sign(y(\theta))$ and $\sign(y(\phi))$ are not roughly proportional to each other, the matrix $\begin{pmatrix}\alpha(\theta)&\beta(\theta)\\ \alpha(\phi)&\beta(\phi)\\ \end{pmatrix}$ is well-invertible, which allows us to deduce from the approximate matrix equation that $Py(\theta)$ and $Qy(\theta)$ can both be approximated by linear combinations of $n^{-1/2}\sign(y(\theta))$ and $n^{-1/2}\sign(y(\phi))$, with coefficients that are not too large. Therefore, $y(\theta)$ can as well, which implies that $x(\theta)$ can be. But $\sign(y(\theta))$ and $\sign(y(\phi))$ take at most two values each on almost all of $E$, and $x(\theta)$ is small outside $E$, so $x(\theta)$ can be approximated by a vector whose coordinates have at most five distinct values. Then we can obtain a contradiction from Lemma \ref{nobadpoints}. 

To carry out this argument we begin by making precise the statement that $\sign(y(\theta))$ and $\sign(y(\phi))$ are not roughly proportional.

\begin{lemma}\label{twosignvectors}
Let $u$ and $v$ be vectors in $\R^n$ that take values $\pm n^{-1/2}$ in a set $E$ of size $m$. Suppose that there are $r$ values in $E$ with $u_i=v_i$ and $s$ values with $u_i\ne v_i$. Then for every $\lambda\in\R$, $|u-\lambda v|\geq 2(rs/mn)^{1/2}$. 
\end{lemma}

\begin{proof}
We have that
\begin{align*}n|u-\lambda v|^2&\geq r(1-\lambda)^2+s(1+\lambda)^2\\
&=(1+\lambda^2)m+2\lambda(s-r).
\end{align*}
This is minimized when $\lambda=(r-s)/m$, and the minimum works out to be $4rs/m$. The lemma follows on dividing both sides by $n$ and taking the square root.
\end{proof}

We showed earlier that the distance between $n^{-1/2}P_E\sign(x(\theta))$ and $\pm n^{-1/2}P_E\sign(x(\phi))$ is at least $\kappa$, which by assumption is at least $\be$. 
It follows further that the conditions of Lemma \ref{twosignvectors} apply to $n^{-1/2}\sign(x(\theta))$ and $n^{-1/2}\sign(x(\phi))$ with both $r$ and $s$ at least  $\be^2 n/4$. 
Therefore, using the trivial bound $|E|\leq n$, we deduce that \[n^{-1/2}|\sign(x(\theta))-\lambda\sign(x(\phi))|\geq \be^2/2.\]

Therefore, using the fact that
\[n^{-1/2}|\sign(x(\theta))-\sign(y(\theta))|\leq\be^2/8\]
and
\[n^{-1/2}|\sign(x(\phi))-\sign(y(\phi))|\leq\be^2/8,\]
we find that
\[n^{-1/2}|\sign(y(\theta))-\lambda\sign(y(\phi))|\geq \frac{\be^2}{2}-\frac{\be^2}{8}(|\lambda|+1) \ge \be^2/8\] 
when $|\lambda|\leq 2$. 
In case $|\lambda|\geq 2$ we can instead use the bound 
\[(1+\lambda^2)m+2\lambda(s-r)\geq(|\lambda|-1)^2m\]
to deduce that 
\[n^{-1/2}|\sign(x(\theta))-\lambda\sign(x(\phi))|\geq \frac{\be^2}{2} (|\lambda|-1),\]
from which it follows that
\[n^{-1/2}|\sign(y(\theta))-\lambda\sign(y(\phi))|\geq \frac{\be^2}{2}(|\lambda|-1)-\frac{\be^2}{8}(|\lambda|+1),\]
which is again at least $\beta^2/8$.

Now that we have shown in a precise sense that $n^{-1/2}\sign(y(\theta))$ and $n^{-1/2}\sign(y(\phi))$ are not approximately proportional to each other, we turn to deducing that the matrix $\begin{pmatrix}\alpha(\theta)&\beta(\theta)\\ \alpha(\phi)&\beta(\phi)\\ \end{pmatrix}$ is well-invertible, by which we simply mean that its determinant is not too small.

\begin{lemma}\label{findlambda}
Let $A=\begin{pmatrix}a&b\\ c&d\\ \end{pmatrix}$ be a $2\times 2$ real matrix,  let $x$ and $y$ be vectors in a Euclidean space such that $\langle x,y\rangle=0$, and let
\[\begin{pmatrix}u\\ v\\ \end{pmatrix}=\begin{pmatrix}a&b\\ c&d\\ \end{pmatrix}\begin{pmatrix}x\\ y\\ \end{pmatrix}.\]
Then there exists $\lambda$ such that $|u-\lambda v|\leq\frac{|x|\,|y|\,|\det(A)|}{|v|}$
\end{lemma} 

\begin{proof}
Consider first the case where $x$ and $y$ are unit vectors. Then 
\begin{align*}|u-\lambda v|^2&=(a-\lambda c)^2+(b-\lambda d)^2\\
&=(c^2+d^2)\lambda^2-2(ac+bd)\lambda+a^2+b^2.
\end{align*}
This is minimized when $\lambda=\frac{ac+bd}{c^2+d^2}$, and the minimum is
\[a^2+b^2-\frac{(ac+bd)^2}{c^2+d^2}=\frac{(ad-bc)^2}{c^2+d^2}\ ,\]
which proves the result. 

In the general case, we have that
\[\begin{pmatrix}a&b\\ c&d\\ \end{pmatrix}\begin{pmatrix}x\\ y\\ \end{pmatrix}=\begin{pmatrix}a|x|&b|y|\\ c|x|&d|y|\\ \end{pmatrix}\begin{pmatrix}x/|x|\\ y/|y|\\ \end{pmatrix}\ .\]
Using the case for unit vectors, we deduce that there exists $\lambda$ such that 
\[|u-\lambda v|^2\leq \frac{|x|^2|y|^2\det(A)^2}{c^2|x|^2+d^2|y|^2},\]  
and again the result is proved.
\end{proof}

Let us now apply this lemma with  $A=\begin{pmatrix}\alpha(\theta)&\beta(\theta)\\ \alpha(\phi)&\beta(\phi)\\ \end{pmatrix}, x=Py(\theta)$ and $y=Qy(\theta)$. Let $\begin{pmatrix}u'\\ v'\\ \end{pmatrix}=A\begin{pmatrix}x\\ y\\ \end{pmatrix}$ and let $u=n^{-1/2}\sign(y(\theta)), v=n^{-1/2}\sign(y(\phi))$. The approximate matrix equation proved earlier states that $|u-u'|\le \sigma$ and $|v-v'|\le \sigma$.
It follows from the lemma that 
there exists $\lambda$ such that 
\[\frac{|Py(\theta)|\,|Qy(\theta)|\,|\det(A)|}{|v'|}\ge|u'-\lambda v'| \ge |u-\lambda v| -(1+|\lambda|)\sigma.
\]
But we have shown that $n^{-1/2}|\sign(y(\theta))-\lambda\sign(y(\phi))|\ge\be^2/8$ for every $\lambda$. Since we also know that $|Py(\theta)|\leq 1$ and $|Qy(\theta)|\leq 1$, it follows that
\[|\det(A)| \ge (\be^2/8-(1+|\lambda|)\sigma)|v'|.
\]
From the proof of the last lemma it follows that the minimum distance is achieved when $\lambda=\frac{\alpha(\phi) \alpha(\theta)|x|^2+\beta(\phi) \beta(\theta) |y|^2}{|v'|^2}$. Recall also from the beginning of the section that $\alpha(\phi),\alpha(\theta),\beta(\phi),\beta(\theta) \le 2\eta^{-1}$. Since $|v'-v|\leq\sigma$, $v=n^{-1/2}\sign(y(\phi))$, and every $y(\phi)$ has support size at least $n/4$, we get that $|v'|\ge 1/2-\sigma$.

Hence 
\[|\lambda| \le \frac{4}{\eta^2(1/2-\sigma)^2}.
\]
 Assuming that 
\begin{align}
    \sigma\leq\be^2\eta^2/2^{10}
\end{align} 
we may deduce that $\sigma(1+|\lambda|)\leq\be^2/16$ and hence that $|\det A|\geq(1/2-\sigma)\be^2/16\geq\be^2/64$.


Let us rewrite the approximate matrix equation as
\[\begin{pmatrix}\alpha(\theta)&\beta(\theta)\\ \alpha(\phi)&\beta(\phi)\\ \end{pmatrix}\begin{pmatrix}Py(\theta)\\ Qy(\theta)\\ \end{pmatrix}=\begin{pmatrix}u'\\ v'\\ \end{pmatrix}\approx_\sigma n^{-1/2}\begin{pmatrix}\sign(y(\theta))\\ \sign(y(\phi))\end{pmatrix}.\]
Then
\[\begin{pmatrix}Py(\theta)\\ Qy(\theta)\\ \end{pmatrix}=\det(A)^{-1}\begin{pmatrix}\beta(\phi)& -\beta(\theta)\\ -\alpha(\phi)&\alpha(\theta)\\ \end{pmatrix}\begin{pmatrix}u'\\ v'\\ \end{pmatrix}.\]

Since the coefficients of $A$ have absolute value at most $2\eta^{-1}$,  it follows that both $Py(\theta)$ and $Qy(\theta)$ are linear combinations of $u'$ and $v'$ with coefficients of absolute value at most $128 \eta^{-1}\be^{-2}$. 
Using again the fact that $|u'-n^{-1/2}\sign(y(\theta))|$ and $|v'-n^{-1/2}\sign(y(\phi))|$ are both at most $\sigma$, it follows that both $Py(\theta)$ and $Qy(\theta)$ can be approximated to within $256 \eta^{-1}\beta^{-2}\sigma$ by the corresponding linear combinations of $n^{-1/2}\sign(y(\theta))$ and $n^{-1/2}\sign(y(\phi))$, and hence that $y(\theta)$ can be approximated to within $512\eta^{-1}\beta^{-2}\sigma$ by a linear combination of $n^{-1/2}\sign(y(\theta))$ and $n^{-1/2}\sign(y(\phi))$ with coefficients of absolute value at most $256\eta^{-1}\beta^{-2}$. 

Now recall that 
\[n^{-1/2}|P_E\sign(x(\theta))-P_E\sign(y(\theta))|\leq 2(c+\xi^{-2}\d^2)^{1/2},\]
and similarly for $\phi$, which implies in particular that $n^{-1/2}P_E\sign(y(\theta))$ and $n^{-1/2}P_E\sign(y(\phi))$ can be approximated to within $2(c+\xi^{-2}\d^2)^{1/2}$ by vectors whose coordinates take just the values $\pm n^{-1/2}$ on $E$. This is because $x(\theta)$ and $x(\phi)$ are typical, which implies, by the second condition in the definition of ``typical", that all coordinates of $P_E \sign(x(\theta))$ and $P_E \sign(x(\phi))$ are $\pm 1$.  
 
Putting together the bounds obtained in the last two paragraphs we get that $P_E y(\theta)$ can be approximated by a linear combination of $n^{-1/2}P_E \sign(x(\theta))$ and  $n^{-1/2}P_E \sign(x(\phi))$ to within $512\eta^{-1}\beta^{-2}\sigma + 1024 \eta^{-1}\beta^{-2} (c+\xi^{-2}\delta^2)^{1/2}$. Thus, $P_E y(\theta)$ can be approximated by a vector with at most four distinct coordinates.

 This in turn implies that $P_Ex(\theta)$ can be approximated to within 
$\delta +512 \eta^{-1}\beta^{-2}\sigma + 1024 \eta^{-1}\beta^{-2}(c+\xi^{-2}\delta^2)^{1/2}$ by such a vector.  But $|x(\theta)-P_Ex(\theta)|\leq\a$, 
 so we end up with the conclusion that $x(\theta)$ can be approximated to within 
$\alpha+\delta +512 \eta^{-1}\beta^{-2}\sigma + 1024\eta^{-1}\beta^{-2}(c+\xi^{-2}\delta^2)^{1/2}$ by a vector that takes at most five distinct values (the fifth value being zero).  But $x(\theta)\in Y$, so it is an $\e$-good point, which implies by Lemma \ref{closetoespace} that $d(x(\theta),PX)\leq 3\d+2\eta$ or $d(x(\theta),QX)\leq 3\d+2\eta$.
 Provided we have chosen our parameters in such a way that 
 \begin{align}\label{cond11}
    \alpha+4\delta +2\eta+512 \eta^{-1}\beta^{-2}\sigma + 1024\eta^{-1}\beta^{-2}(c+\xi^{-2}\delta^2)^{1/2} \le \gamma,
 \end{align} 
this contradicts the fact that $P$ was chosen to ensure that neither $PX_\g$ nor $QX_\g$ contains a vector with at most five distinct coordinates (see the end of Section \ref{outline} where this was shown to be possible).
 
\subsection{Case 3: $k\geq 5$.}

We begin with a simple lemma to estimate how well we can simultaneously approximate $k$ orthonormal vectors by a $(k-1)$-dimensional subspace.

\begin{lemma}\label{approxon}
Let $W$ be a $(k-1)$-dimensional subspace of $\R^n$ and let $u_1,\dots,u_k$ be an orthonormal sequence. Then there exists $i$ such that $d(u_i,W)\geq k^{-1/2}$. 
\end{lemma} 

\begin{proof}
Without loss of generality $n=k$. Now let $v$ be a unit vector orthogonal to $W$. Then the orthogonal projection $P_W$ to $W$ is given by the formula $P_W(x)=x-\langle x,v\rangle v$, from which it follows that $d(x,W)=|\langle x,v\rangle|$. But $\sum_{i=1}^k\langle u_i,v\rangle^2=1$, so there must exist $i$ such that $|\langle u_i,v\rangle|\geq k^{-1/2}$, which proves the lemma.
\end{proof}

Now, with the help of the assumption that $k\geq 5$, we prove that we cannot find a 4-dimensional subspace that approximately contains all the vectors in $\Sigma$. For convenience, let us reorder the coordinates in such a way that $E=\{1,2,\dots,m\}$ and $0\leq\phi_1\leq\phi_2\leq\dots\leq\phi_m<2\pi$. Then for every typical vector $x(\theta)$, the set of $i$ such that $x(\theta)_i>0$ is an interval mod $m$.

\begin{lemma}\label{distance}
Let $W$ be a 4-dimensional subspace of $X$. Then $\Sigma$ contains a vector $u$ such that $d(u,W)\geq\be/2\sqrt 5$.
\end{lemma} 
\begin{proof}
Let $0\leq\theta_1<\dots<\theta_5<\pi$ be such that the points $u_j=n^{-1/2}P_E\sign(x(\theta_j))\in\Sigma$ and that together with the points $-u_j$ form a $\be$-separated subset of $\Sigma$. For each $j\in \{1,2,\dots,5\}$ let $[a_j,b_j]$ be the interval mod $m$ of $i$ such that $(u_j)_i=n^{-1/2}$. 

Note that $a_1\geq\dots\geq a_5$ and $b_1\geq\dots\geq b_5$, where here we refer to the cyclic ordering on the integers mod $m$. It follows that for each $j$, $(u_{j+1}-u_j)_i= 2n^{-1/2}$ on the interval $[a_{j+1},a_j)$, $- 2n^{-1/2}$ on the interval $(b_{j+1},b_j]$, and zero everywhere else. 
In particular, the vectors of the form $u_{j+1}-u_j$ for $j=1,2,3,4$, together with the vector $u_1+u_5$, are orthogonal.  

Since we have that $|u_{i}\pm u_j|\geq\be$ for every $i\neq j$,
setting $v_j=\frac{u_{j+1}-u_j}{|u_{j+1}-u_j|}$ for $j=1,\ldots ,4$ and $v_5=\frac{u_1+u_5}{|u_1+u_5|}$, we may deduce from Lemma \ref{approxon} that $d(v_j,W)\geq 1/\sqrt 5$ for some $j$. If $j=1,\ldots, 4$ then this implies that $d(u_{j+1}-u_j,W)\geq\be/\sqrt 5$, which implies that either $d(u_{j+1},W)\geq\be/2\sqrt 5$ or $d(u_j,W)\geq\be/2\sqrt 5$.  Similarly, if $j=5$ we get that either $d(u_1,W)\ge \be/2\sqrt 5$ or $d(u_5,W)\ge \be/2\sqrt 5$. 
This proves the lemma.
\end{proof}

Let $u=P_E\sign(x(\theta))$ be given by Lemma \ref{distance}. Recall that (with parameters that satisfy condition \eqref{cond11}) we have
\[n^{-1/2}|P_E\sign(x(\theta))-P_E\sign(y(\theta))|\leq\be^2/8.\] 
Recall also that
\[n^{-1/2}\sign(y(\theta))\approx_{\d/\eta}\alpha(\theta) Py(\theta)+\beta(\theta) Qy(\theta)\]
for some coefficients $\alpha(\theta),\be(\theta)$ that have absolute values at most $2\eta^{-1}$. We also know that $|y(\theta)-x(\theta)|\leq\d$. It follows that
\[n^{-1/2}P_E\sign(x(\theta))\approx_{\be^2/8+\d/\eta+4\d/\eta}\alpha(\theta) P_EPx(\theta)+\beta(\theta) P_EQx(\theta).\]

It follows that the distance from $n^{-1/2}P_E\sign(x(\theta))$ to the subspace $P_E(PY+QY)$, which has dimension at most 4, is at most $\be^2/8+5\d/\eta$.  If we pick parameters in such a way that 
\begin{align}\label{cond12}
\frac{\be^2}8+\frac{5\d}{\eta}<\frac{\be}{2\sqrt{5}},
\end{align} 
then this contradicts Lemma \ref{distance}.

\subsection{Choosing parameters}
We conclude by showing that there exists a choice of parameters which fulfils all the conditions.  We are not optimizing this choice. 

First, recall that we have already chosen $\gamma$ in the Corollary \ref{pickgamma} to be $2^{-37}$. Further we see that $\beta=2^{-6}$ satisfies \eqref{cond5} and \eqref{cond8}. We can further choose $\eta=2^{-40}$, $\alpha=2^{-40}$, $\rho=2^{-40}$ and then $c=2^{-205}$, $\xi=2^{-403}$ and finally $\delta=2^{-506}$. It is easy to check that these parameters meet all the conditions.

This finishes the proof that the $2$-Euclidean norm defined in \eqref{definition of the norm}, for $n\ge 35$ (which comes from our choice of $\gamma$ and the condition in Lemma \ref{subspacevol} and the fact that we use a $2\gamma$-expansion), contains no 2-dimensional subspace which is both strongly $(1+\epsilon)$-complemented and strongly $(1+\epsilon)$-Euclidean with $\epsilon=\delta^2/8C^2 =2^{-1017}$. (the dependence of $\e$ on $\delta$ was established in the proof of Lemma \ref{goodpoints}). 



\nocite{*}
\bibliography{bibfile} 

\begin{thebibliography}{1}

\bibitem{Dvoretzky1961}
A.~Dvoretzky.
\newblock Some results on convex bodies and {B}anach spaces.
\newblock In {\em Proc. Internat. Sympos. Linear Spaces (Jerusalem, 1960)},
  pages 123--160. Academic Press, 1961.

\bibitem{Milman1971}
V.~Milman.
\newblock {New proof of the theorem of {A}. {D}voretzky on intersections of
  convex bodies}.
\newblock {\em Functional Analysis and Its Applications}, 5(4):288--295, 1971.

\bibitem{szarektomczak2009}
S.~Szarek and N.~Tomczak-Jaegermann.
\newblock On the nontrivial projection problem.
\newblock {\em Advances in Mathematics}, 221:331--342, 2009.

\end{thebibliography}
\bibliographystyle{abbrv}

\end{document}